\documentclass{article}

\usepackage{arxiv}

\usepackage[utf8]{inputenc} 
\usepackage[T1]{fontenc}    
\usepackage{hyperref}       
\usepackage{url}            
\usepackage{booktabs}       
\usepackage{amsfonts}       
\usepackage{nicefrac}       
\usepackage{microtype}      

\usepackage{graphicx}
\usepackage{dcolumn}
\usepackage{bm}

\usepackage[utf8]{inputenc}
\usepackage[T1]{fontenc}
\usepackage{mathptmx}

\usepackage{amsmath}
\usepackage{amsthm}
\usepackage{amsfonts}
\usepackage{amssymb}
\def \bes{\begin{eqnarray}}
\def \ees{\end{eqnarray}}
\def \bns{\begin{eqnarray*}}
	\def \ens{\end{eqnarray*}}

\newtheorem{remark}{Remark}

\newtheorem{theorem}{Theorem}[section]

\newtheorem{lemma}{Lemma}[section]

\usepackage{multirow}

\title{Bifurcation analysis of a free boundary model of plaque formation associated with the cholesterol ratio}

\author{
 Wenrui Hao\\
Department of Mathematics\\
Pennsylvania State University\\
State College, PA 16802, USA \\
  \texttt{wxh64@psu.edu} \\
\And
Chunyue Zheng \\
Department of Mathematics\\
Pennsylvania State University\\
State College, PA 16802, USA \\
  \texttt{cmz5199@psu.edu} \\
}

\begin{document}
\maketitle

\begin{abstract}
The low-density lipoprotein (LDL)/high-density lipoprotein (HDL)-cholesterol ratio has been shown a high correlation with the cardiovascular risk assessment. Is it possible to quantify the correlation mathematically? In this paper, we develop a bifurcation analysis for a mathematical model of the plaque formation with a free boundary in the early stage of atherosclerosis. This bifurcation analysis, to the ratio of LDL/HDL, is based on the explicit formulation of radially symmetric steady-state solutions. By performing the perturbation analysis to these solutions, we establish the existence of bifurcation branches and derive a theoretical condition that a bifurcation occurs for different modes. Then we also analyze the stability of radially symmetric steady-state solutions and conduct numerical simulations to verify all the theoretical results.
\end{abstract}

	\section{Introduction}\label{sec:intro}

Atherosclerosis, known as an inflammatory disease\cite{ross1999atherosclerosis,stoll2006inflammation}, is the No.1 killer of Americans. It can affect any artery in the body and most of those deaths are from heart attacks caused by fatty deposits that clog coronary arteries. These deposits, which are called plaques, consist of cholesterol, fat, and other substances\cite{childbirthcarotid}.
As the plaque builds up, the artery wall gets thicker, which narrows the blood vessel and reduces the supply of oxygen to cells. Then the plaque may rupture and the bloodstream would carry the debris until it gets stuck, leading to the formation of thrombus. The arteries can be blocked during this process and heart attacks or strokes may occur, depending on where the plaque locates\cite{moreno2010vulnerable}.

The arterial wall usually consists of three layers: the intima, media, and adventitia (see Fig. \ref{fig:domain}). The intima is a thin single sheet of endothelial cells. The media is composed mainly of smooth muscle cells and elastic tissue. The adventitia is the outermost connective tissue layer\cite{drake2009gray, steve2003layers}.
The development of plaque begins with a lesion in the intima layer, initiating an inflammatory response resulting in the accumulation of LDLs\cite{bentzon2014mechanisms}. Part of LDLs become oxidized LDLs by free radicals and would be ingested by macrophages differentiating from monocytes. The ingestion of large amounts of oxidized LDLs transforms macrophages into foam cells that are responsible for plaque growth. In the meanwhile, the HDLs remove cholesterol from foam cells and inhibit the oxidation of LDLs\cite{barter2005role}. Therefore the balance of HDLs and LDLs is essential for plaque development. According to the cholesterol guideline of the AHA \cite{LDLHDL2,LDLHDL}, the optimal cholesterol ratio, LDL/HDL, is 3.5.
A higher ratio means a higher cardiovascular risk. In other words, individuals who have a higher ratio need to work toward the optimal ratio, either by changing their lifestyles or by eating heart-healthy diets, to reduce the cardiovascular risk. In this paper, we will interpret the importance of the LDL/HDL ratio in a mathematical modeling context. 	

Several mathematical models have been developed to explore the relationship between cholesterol ratio and cardiovascular risk \cite{calvez2009mathematical,cobbold2002lipoprotein,mckay2005towards,friedman2015mathematical}.  These models characterize biological interactions among endothelial cells, monocytes, and T cells by using partial differential equations (PDEs) and address the importance of LDL and HDL in plaque growth. Among these mathematical models, some of them are free boundary problems to describe the geometric change of the plaque in the artery \cite{HF,friedman2015mathematical,FHH}. For instance, a recent free boundary model \cite{HF} introduces a system of PDEs including LDL, HDL, macrophages, T cells, smooth muscle cells, and related cytokines and generates a ``risk-map" of plaque development for any pair values of (LDL, HDL), indicating the significance of  LDL and HDL in determining the growth or shrink of a plaque. Later, the effect of reverse cholesterol transport (RCT) has been added to this free boundary model \cite{friedman2015mathematical}. Moreover, a simplified free boundary model has been analyzed theoretically on the existence of small radially symmetric stationary plaques and their stability conditions\cite{FHH}. However, there is no theoretical analysis of the effect of cholesterol ratio on plaque growth for these free boundary models.

In this paper, we develop a free boundary model of plaque growth in the early stage of atherosclerosis and theoretically analyze the bifurcation to cholesterol ratio.
The paper is organized as follows: In Section \ref{sec:model}, we introduce a  mathematical model of plaque formation with a free boundary and derive the explicit formula of radially symmetric steady-state solutions; In Section \ref{sec:nonradially-steady}, we establish the existence of bifurcation branches from radially symmetric steady-state solutions, obtain a theoretical condition that the bifurcation occurs, and explore the linear stability of radially symmetric steady-state solutions; In Section \ref{sec:numerical}, we conduct numerical simulations of the free boundary problem and verify all the theoretical results.

\begin{figure}[htbp]
	\centering
	\includegraphics[width=0.5\linewidth]{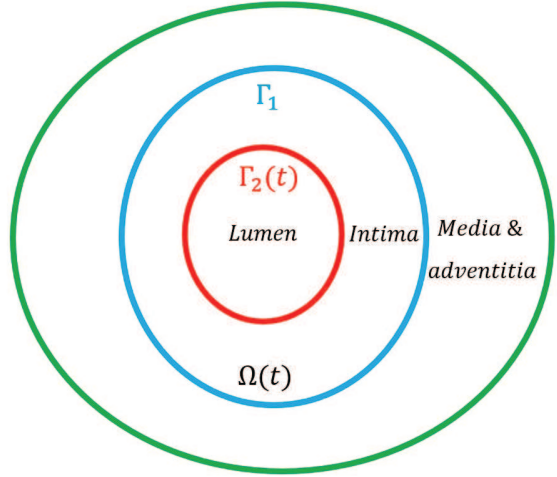}
	\caption{The domain of the free boundary model: $\Omega(t)$ represents the intima; the inner surface of the arterial wall, $\Gamma_2(t)$, is a free boundary; and the surface between the intima and media/adventitia, $\Gamma_1$, is fixed.}
	\label{fig:domain}
\end{figure}

\section{Mathematical model}\label{sec:model}
We consider plaque formation in the early stage of atherosclerosis by including the basic pathophysiology in the intima (See Fig. \ref{fig:domain} for the detailed domain setup). Macrophages enter the intima, $\Omega(t)$, by the chemotaxis of MCP-1 \cite{pmid2052604} and become foam cells by the uptake of oxidized LDL \cite{pmid32193098}. On the other hand, HDL removes the cholesterol from foam cells \cite{pmid30118702} which become M2 macrophages transfer back to the liver, referred to as the RCT process \cite{pmid31317009}. Then we model the density of macrophages, $M$, below:
\begin{eqnarray}\left\{
\begin{array}{ll}
	\displaystyle\frac{\partial M}{\partial t}-D\Delta M  = -HM , &
	x\in\Omega(t),\\
	\displaystyle  \frac{\partial M}{\partial \mathbf{n}}  = -M,  &
	x\hbox{~on~}\Gamma_1,\\
	\displaystyle  M = 1,&
	x\hbox{~on~}\Gamma_2(t),\\
\end{array}\right.\label{MF}\end{eqnarray}
where $H$ represents the concentration of HDL and $-HM$ accounts for the loss of macrophages due to the RCT process \cite{HF}. On $\Gamma_2(t)$, we use the Dirichlet boundary condition to model the recruitment of macrophages by MCP-1 and take $M=1$ after the normalization. Since there are no macrophages in media/adventitia \cite{HF}, we have $\displaystyle  \frac{\partial M}{\partial \mathbf{n}}  +\alpha M=0$ on $\Gamma_1$. For simplicity, we take the flux rate $\alpha=1$ in our model.

Plaque growth is proportional to the density of foam cells which is assumed to be a combination of LDL and macrophages in our model. Therefore we model plaque growth as
\begin{eqnarray}
\nabla\cdot\mathbf{v}  = LM - T, &\
x\in\Omega(t),
\end{eqnarray}
where $L$ represents the concentration of LDL and $T$ is the clearance capacity provided by the immune system\cite{libby2002inflammation}: if there is too much LDL and macrophages ($LM>T$), the plaque will grow; otherwise will disappear due to the immune system.  For simplicity, we treat $T$ as a parameter in our model instead of including its dynamics.

In light of the intima's high permeability to white cells and platelets \cite{CMP,CPM,GKBVBS}, we treat the intima as a porous medium and macrophages in the intima as a low-speed flow \cite{friedman2015mathematical,HF}  moving with a common velocity, $\mathbf{v}$. Thus the pressure, $P$, resulting from the movement of macrophages, follows Darcy's law $\mathbf{v}=-\nabla P$. Therefore the equation of $P$ \cite{friedman2015mathematical,HF} becomes: \begin{eqnarray}\left\{
\begin{array}{ll}
	-\displaystyle \Delta P = LM - T , &
	x\in\Omega(t),\\
	\displaystyle  \frac{\partial P}{\partial \mathbf{n}}=0& x\hbox{~on~} \Gamma_1,\\
	\displaystyle P = \gamma \kappa, ~\mathbf{v}_\mathbf{n} (t) =  - \frac{\partial P}{\partial \mathbf{n}}& x\hbox{~on~} \Gamma_2(t). \label{P} %
\end{array}\right.
\end{eqnarray}
Since the boundary $\Gamma_1$ is fixed, we have the no-flux boundary condition. On $\Gamma_2(t)$, the pressure $P$ is balanced by the surface-tension which is proportional to the mean curvature $\kappa$ ($\gamma$ is the proportionality constant, the blood pressure is considered as zero in our model for simplicity); the velocity along the normal direction $\mathbf{n}$, $\mathbf{v}_\mathbf{n}=\mathbf{v}\cdot\mathbf{n}$ gives the free boundary moving condition \cite{HF}. Thus we summarize the free boundary model as follows
\begin{eqnarray}\left\{
\begin{array}{ll}
	\displaystyle\frac{\partial M}{\partial t}-D\Delta M  = -HM , &
	x\in\Omega(t),\\
	-\displaystyle \Delta P = LM - T , &
	x\in\Omega(t),\\
	\displaystyle  \frac{\partial P}{\partial \mathbf{n}}=0,		\displaystyle  \frac{\partial M}{\partial \mathbf{n}}  = -M& x\hbox{~on~} \Gamma_1,\\
	\displaystyle  M = 1,		\displaystyle P = \gamma \kappa, ~\mathbf{v}_\mathbf{n} (t) =  - \frac{\partial P}{\partial \mathbf{n}}& x\hbox{~on~} \Gamma_2(t). \label{sys} %
\end{array}\right.
\end{eqnarray}
\subsection{Explicit formula of radially symmetric steady-sate solutions}	

 The 2D polar coordinate system: We first introduce the 2D polar coordinate with the radial coordinate $r$ and the angular coordinate $\theta$. Then $M$ and $P$ are functions of $(r,\theta,t)$ and the free boundary is represented as $r-\rho(\theta,t)=0$.
	Specifically, in the radially symmetric case, since $\frac{\partial M}{\partial \theta}=\frac{\partial P}{\partial \theta}=0$, we have $M(r,t)$  and $P(r,t)$. Moreover, the free boundary is $\Gamma_2 = \{r|r = \rho(t)\}$ and the fixed boundary is $\Gamma_1 = \{r|r = R\}$. Moreover, the steady-state solutions denote as $M(r)$ and $P(r)$ ($\rho\leq r\leq R$) since $t\rightarrow\infty$.

{\bf Radially symmetric solution of $M$:}    First we compute the radially symmetric steady-state solution of \eqref{MF} by taking $\frac{\partial M}{\partial t}=0$ and have
\begin{equation}
\label{eqn:M}
\left\{
\begin{aligned}
&M''(r)+\frac{1}{r}M'(r)-\frac{H}{D}M=0,\\
&\frac{\partial M}{\partial r}|_{r=R}  = -M(R),\\
&M(\rho) = 1.\\
\end{aligned}
\right.
\end{equation}
By taking $z=\sqrt{\frac{H}{D}}r$, we rewrite \eqref{eqn:M} in terms of $u(z)=M(r)$ as
\begin{equation*}
z^2u''(z)+zu'(z)-z^2u(z)=0,
\end{equation*}
which implies
\begin{equation}
\label{sol:M_S}
M_s(r)=u(z)=C_1I_0(z_r)+C_2K_0(z_r).
\end{equation}
Since $I_0'(z)=I_1(z)$ and $K_0'(z)=-K_1(z)$, we solve for $C_1$ and $C_2$ by using the boundary conditions:
\begin{equation}
\left\{
\begin{aligned}
&C_1\sqrt{\frac{H}{D}}I_1(z_R)-C_2\sqrt{\frac{H}{D}}K_1(z_R)=-(C_1I_0(z_R)+C_2K_0(z_R)),\\
&C_1I_0(z_\rho)+C_2K_0(z_\rho) = 1.
\end{aligned}
\right.
\end{equation}
Then we have
\begin{equation*}
\begin{aligned}
C_1=\frac{\sqrt{\frac{H}{D}}K_1(z_R)-K_0(z_R)}{C(\rho,R,H,D)}\hbox{~and~} C_2 =  \frac{\sqrt{\frac{H}{D}}I_1(z_R)+I_0(z_R)}{C(\rho,R,H,D)},
\end{aligned}
\end{equation*}
where
\begin{equation*}
\begin{aligned}
C(\rho,R,H,D)=&(I_0(z_R)K_0(z_\rho)-I_0(z_\rho)K_0(z_R))\\
&+\sqrt{\frac{H}{D}}(K_1(z_R)I_0(z_\rho)+I_1(z_R)K_0(z_\rho))
\end{aligned}
\end{equation*}
{\bf Remark:} By the maximum principle, we have $M_s(r)\ge0$ for $\rho\le r\le R$,  $M'_s(\rho) < 0$, and $M'_s(R) < 0$.

{\bf Radially symmetric solution of $P$:}	By rewriting \eqref{P} as $ \Delta (P +\frac{DL}{H}M)=  T $, we have
\begin{equation}
\label{sol:P_S}
P_s(r)=-\frac{DL}{H}M_s(r)+C_3\ln r+C_4 + \frac{1}{4}Tr^2.
\end{equation}
The boundary conditions in radially symmetric case become
\begin{equation}
P_s(\rho)=\frac{\gamma}{\rho},\quad \frac{\partial P_s}{\partial r}|_{r=R}=0, \hbox{~and~}\frac{\partial P_s}{\partial r}|_{r=\rho}=0,
\end{equation}
which are used to determine $C_3$, $C_4$, and $T$, namely,
\begin{equation*}
C_3=\frac{DL}{H}\frac{R^2\rho^2}{R^2-\rho^2}(\frac{M'_s(\rho)}{\rho}+\frac{M_s(R)}{R}),
\end{equation*}
\begin{equation*}
C_4=\frac{\gamma}{\rho}+\frac{DL}{H}-C_3\ln \rho-\frac{1}{4}T\rho^2,
\end{equation*}
and
\begin{equation}
T = -\frac{DL}{H}\frac{2}{R^2-\rho^2}(RM_s(R)+\rho M'_s(\rho)).
\label{eqn:T}
\end{equation}

For any given $T$, we compute $\rho$ by solving (\ref{eqn:T}). Therefore, the existence of $\rho$ is critical for our model. In order to prove the existence, we solve $T$ for any given $\rho$ and have the following theorem.

\begin{theorem}
	For any given $L>0$ and $\rho > 0$, there exists a unique $T>0$ such that a stationary solution $(M_s,P_s)$ is given by \eqref{sol:M_S} and \eqref{sol:P_S}.
\end{theorem}
\begin{proof}
	For any given $\rho$, it is obvious that $T$ is uniquely determined by \eqref{eqn:T}. Next we prove $T>0$ by letting
	$$
	f(r) = rM'_s(r).
	$$
	Since		$$
	f'(r) = rM''_s(r) + M'_s(r) = \frac{H}{D}rM_s(r) \ge 0,
	$$
	we have
	\begin{equation*}
	f(\rho) = \rho M'_s(\rho) \le f(R) = R M'_s(R) = -RM_s(R)
	\end{equation*}
	which implies
	$$
	T= -\frac{DL}{H}\frac{2}{R^2-\rho^2}(f(\rho) - f(R)) \ge 0.
	$$
	
\end{proof}

\section{Bifurcation analysis and linear stability}\label{sec:nonradially-steady}
\subsection{The linearized system}
\label{subsec:linearized}
First, we derive the linearized system of (\ref{sys}) 	
with a perturbed domain $\Omega_\varepsilon$ to $\Omega$, namely, $\Gamma_{\varepsilon}=\{ r|r = \rho +\varepsilon\rho_1(\theta)\}$:
\begin{eqnarray}\left\{
\begin{array}{ll}
	\displaystyle-D\Delta M  = -HM , &
	x\in\Omega_\varepsilon,\\
	\displaystyle  \frac{\partial M}{\partial \mathbf{n}}  = -M,  &
	x\hbox{~on~}\Gamma_1,\\
	\displaystyle  M = 1,&
	x\hbox{~on~}\Gamma_{\varepsilon},\\
	-\displaystyle \Delta P = LM - T , &
	x\in\Omega_\varepsilon,\\
	\displaystyle  \frac{\partial P}{\partial \mathbf{n}}=0& x\hbox{~on~} \Gamma_1,\\
	\displaystyle P = \gamma \kappa,& x\hbox{~on~} \Gamma_{\varepsilon}.
\end{array}\right.\label{eqn:perturbed}\end{eqnarray}

By defining the following nonlinear function $F$ based on the free boundary condition,
\begin{eqnarray}
F(\rho_1,L)=\frac{\partial P}{\partial r}\vline_{\Gamma_\varepsilon}\label{eqn:nonlinear},
\end{eqnarray}
we conclude that $\rho_1(\theta)$ induces a stationary solution if and only if $F(\rho_1,L)=0$. Then we consider the solution of (\ref{eqn:perturbed}), $(M,P)$,
up to the 2nd order of $\varepsilon$:
\begin{equation}
\begin{aligned}
&M(r,\theta) = M_s(r) +\varepsilon M_1(r,\theta) +\mathcal{O}(\varepsilon^2),\\
&P(r,\theta) = P_s(r) +\varepsilon P_1(r,\theta) +\mathcal{O}(\varepsilon^2).\\
\end{aligned}
\label{eqn:expansion}
\end{equation}
Thus the boundary condition of $M$ on $\Gamma_\varepsilon$ becomes
\begin{equation*}
\begin{aligned}
1 &= M(r,\theta)|_{\Gamma_{\varepsilon}} = M(\rho +\varepsilon\rho_1,\theta) \\
&= M_s(\rho +\varepsilon\rho_1 +\varepsilon M_1(\rho +\varepsilon\rho_1,\theta) +\mathcal{O}(\varepsilon^2) \\
&= M_s(\rho) + \varepsilon\rho_1\frac{\partial M_s}{\partial r}(\rho) + \varepsilon M_1(\rho) +\mathcal{O}(\varepsilon^2).
\end{aligned}
\end{equation*}

Since the mean curvature is given by
\begin{equation}
\kappa =\frac{2r_\theta^2-r r_{\theta\theta}+r^2}{(r_\theta^2+r^2)^{3/2}},
\end{equation}
the linearization of $\kappa$ becomes
\begin{equation}
\begin{aligned}
\label{eqn:curvature}
\kappa|_{\Gamma_\varepsilon} =& \frac{(\rho+\varepsilon\rho_1)^2+2(\rho_{\theta} +\varepsilon\rho_{1\theta})^2}{((\rho +\varepsilon\rho_1)^2+(\rho_{\theta} +\varepsilon\rho_{1\theta})^2)^{3/2}} \\
& -\frac{(\rho +\varepsilon\rho_1)(\rho_{\theta\theta} +
	\varepsilon\rho_{1\theta\theta})}{((\rho +\varepsilon\rho_1)^2+(\rho_{\theta} +\varepsilon\rho_{1\theta})^2)^{3/2}}\\
=& \kappa_0 + \varepsilon\kappa_1 +  \mathcal{O}(\varepsilon^2),
\end{aligned}
\end{equation}
where
\begin{equation}
\kappa_0 = \frac{2\rho_{ \theta}^2-\rho  \rho_{\theta\theta}+\rho ^2}{(\rho_{ \theta}^2+\rho ^2)^{3/2}}
\end{equation}
and
\begin{equation}
\begin{aligned}
\kappa_1
=& (\frac{2\rho-\rho_{\theta \theta}}{(\rho_{\theta}^2+\rho ^2)^{3/2}}-\frac{3}{2} \frac{(\rho^2+2\rho_{ \theta}^2-\rho\rho_{\theta \theta})2\rho}{(\rho_{ \theta}^2+\rho^2)^{5/2}})\rho_1\\
&+(\frac{4\rho_{ \theta}}{(\rho_{ \theta}^2+\rho^2)^{3/2}}-\frac{3}{2} \frac{(\rho^2+2\rho_{ \theta}^2-\rho\rho_{\theta \theta})2\rho_{\theta}}{(\rho_{ \theta}^2+\rho^2)^{5/2}})\rho_{1 \theta}\\
&-\frac{\rho}{(\rho_{ \theta}^2+\rho ^2)^{3/2}}\rho_{1\theta\theta}.
\end{aligned}
\end{equation}
After dropping the higher order terms, we obtain the linearized system below:
\begin{equation}
\label{eqn:linearized}
\left\{
\begin{aligned}
&\Delta M_1 = \frac{H}{D}M_1\quad\mathrm{in}\quad \Omega,\\
&M_1(\rho) = -\rho_1\frac{\partial M_s(\rho) }{\partial r},\\
&\frac{\partial M_1}{\partial r}(R) = -M_1(R),\\
&\Delta P_1 = -LM_1\quad\mathrm{in}\quad \Omega,\\
&P_1(\rho) = -\frac{\gamma}{\rho^2}(\rho_1 + \rho_{1\theta\theta}),\\
&\frac{\partial P_1}{\partial r}(R)  = 0.
\end{aligned}
\right.
\end{equation}

Assuming $\rho_1(\theta) = \cos(n\theta)$ and, by separation of variables,
\begin{equation}
M_1(r,\theta)=\cos(n\theta)Q_n(r),
\end{equation}
we have
\begin{equation}\label{eqn:Qn}
Q_n(r)=q_n(z) = \tilde C_1 I_n(z_r) + \tilde C_2 K_n(z_r)
\end{equation}
which satisfies $z^2q_n''+zq_n'-(z^2+n^2)q_n=0$, $\frac{\partial Q_n}{\partial r}(R)=-Q_n(R)$, and
$
Q_n(\rho) = -\frac{\partial M_s(\rho) }{\partial r}.
$
Since $I'_n(x)=\frac{n}{x}I_n(x)+I_{n+1}(x)$ and $K'_n(x)=\frac{n}{x}K_n(x)-K_{n+1}(x)$, we have
\begin{equation*}
\begin{aligned}
Q'_n(r)=&\tilde C_1\frac{\partial I_n(z_r)}{\partial r}+\tilde C_2\frac{\partial K_n(z_r)}{\partial r}\\
=&\tilde C_1(\frac{n}{r}I_n(z_r)+\sqrt{\frac{H}{D}}I_{n+1}(z_r))\\
&+\tilde C_2(\frac{n}{r}K_n(z_r)-\sqrt{\frac{H}{D}}K_{n+1}(z_r)),
\end{aligned}
\end{equation*}
where	\begin{equation*}
\left\{
\begin{aligned}
&\tilde C_1 =\frac{-M'_s(\rho)}{I_n(z_\rho)+KK_n(z_\rho)} \\
&\tilde C_2 = \frac{-KM'_s(\rho)}{I_n(z_\rho)+KK_n(z_\rho)}\\
\end{aligned}
\right.
\end{equation*}
and $$
K = -\frac{I_n(z_R)+I'_n(z_R)}{K_n(z_R)+K'_n(z_R)}.
$$
By the maximum principle,  we have $Q_n(r)\ge0$.

Similarly,	we have $P_1(\rho)=\gamma(-\frac{1}{\rho^2}+\frac{n^2}{\rho^2})\cos(n\theta)$ and
$
\frac{\partial P_1}{\partial r}(R)  = 0.
$
Therefore,	we obtain \begin{equation}
P_1 + \frac{LD}{H}M_1 = \tilde C_3 r^n\cos(n\theta) +\tilde C_4 r^{-n}\cos(n\theta),
\end{equation}
where $\tilde C_3 $ and $\tilde C_4 $ satisfy
\begin{equation}
\left\{
\begin{aligned}
&-\frac{LD}{H}\frac{\partial Q_n}{\partial r}(R) + \tilde C_3 nR^{n-1}-\tilde C_4 nR^{-n-1}= 0,\\
&-\frac{LD}{H}Q_n(\rho) + \tilde C_3 \rho^n+\tilde C_4 \rho^{-n} = \gamma(-\frac{1}{\rho^2}+\frac{n^2}{\rho^2}),
\end{aligned}
\right.
\end{equation}
or
\begin{equation}
\left\{
\begin{aligned}
&\tilde C_3 = \frac{\gamma(n^3-n)\rho^{-2+n}-L\frac{D}{H}(Q_n(R)R^{n+1}-nQ_n(\rho)\rho^{n})}{n(\rho^{2n}+R^{2n})},\\
&\tilde C_4 = \tilde C_3 R^{2n}+L\frac{D}{H}Q_n(R)\frac{R^{n+1}}{n}.
\end{aligned}\label{eqn:C4}
\right.
\end{equation}

\subsection{Justification for \eqref{eqn:expansion}}	In this subsection, we justify the validity of expansions in \eqref{eqn:expansion} by showing that the $\mathcal{O}(\epsilon^2)$ terms are small.	First we introduce the following Banach space
\begin{equation*}
\begin{aligned}
X^{l+\alpha} =\  &\{\rho_1\in C^{l+\alpha}: \rho_{1} \mathrm{\ is}\ 2\pi\mbox{-}\mathrm{periodic} \},\\
X^{l+\alpha}_1 =\  &\mathrm{closure\ of\ the\ linear\ space\ spanned}\\
&\mathrm{by\ }\{\cos(j\theta),j=0,1,2,\cdots\}\ \mathrm{in}\ X^{l+\alpha}.
\end{aligned}
\end{equation*}
Then we have the following lemma:	
\begin{lemma}
	If $\rho_{1}\in C^{3+\alpha}(\mathbb{R})$ and $(M,P)$ is the solution of \eqref{eqn:perturbed}, then
	\begin{equation}
	\|M - M_s \|_{C^{3+\alpha}(\bar{\Omega}_\varepsilon)} \le C|\varepsilon|\|\rho_{1}\|_{C^{3+\alpha}(\mathbb{R})},
	\end{equation}
	\begin{equation}
	\|P - P_s \|_{C^{1+\alpha}(\bar{\Omega}_\varepsilon)} \le C|\varepsilon|\|\rho_{1}\|_{C^{3+\alpha}(\mathbb{R})},
	\end{equation}
	where constant $C$ is independent of $\varepsilon$.
\end{lemma}
\begin{proof}
	First we derive the equation of $M-M_s$ below
	\begin{equation}\left\{
	\begin{aligned}
	&\Delta (M-M_s)- \frac{H}{D} (M-M_s) =  0 \quad\mathrm{in}\quad \Omega_\varepsilon,\\
	&\frac{\partial(M-M_s)}{\partial n} + (M-M_s) = 0\quad\mathrm{on}\quad \Gamma_1,\\
	&M-M_s = g_1 \quad\mathrm{on}\quad \Gamma_\varepsilon.\\
	\end{aligned}
	\right.
	\end{equation}
	From \eqref{eqn:perturbed} and $M_s(\rho) = 1$, we have
	\begin{equation}
	\begin{aligned}
	g_1 &= M(\rho + \varepsilon\rho_{1}) - M_s(\rho + \varepsilon\rho_{1}) \\
	&= 1 - M_s(\rho + \varepsilon\rho_{1}) \\
	&= M_s(\rho) - M_s(\rho + \varepsilon\rho_{1}).
	\end{aligned}
	\end{equation}
	By differentiating three times with respect to $\theta$, we obtain		\begin{equation}
	\|M-M_s\|_{C^{3+\alpha}(\Gamma_\varepsilon)} \le C|\varepsilon|\|\rho_{1}\|_{C^{3+\alpha}(\mathbb{R})}.
	\end{equation}
	The Schauder estimates then indicate that
	$$
	\|M - M_s \|_{C^{3+\alpha}(\bar{\Omega}_\varepsilon)} \le C|\varepsilon|\|\rho_{1}\|_{C^{3+\alpha}(\mathbb{R})}.
	$$
	Since $\Gamma_\varepsilon\in C^{3+\alpha}$, the constant $C$ is independent of $\varepsilon$.
	
	Similarly, the equation of $P - P_s$ reads as
	\begin{equation*}\left\{
	\begin{aligned}
	&-\Delta (P-P_s)=  L(M-M_s)\quad\mathrm{in}\quad \Omega_\varepsilon,\\
	&P-P_s = g_2 \quad\mathrm{on}\quad \Gamma_\varepsilon,\\
	&\frac{\partial(P-P_s)}{\partial n}= 0 \quad\mathrm{on}\quad \Gamma_1,\\
	\end{aligned}
	\right.
	\end{equation*}
	where
	\begin{equation*}
	\begin{aligned}
	g_2 &= P(\rho + \varepsilon\rho_{1}) - P_s(\rho + \varepsilon\rho_{1}) = \gamma\kappa|_{r = \rho+\varepsilon\rho_1}- P_s(\rho + \varepsilon\rho_{1}) \\
	&= \frac{\gamma}{\rho} - \gamma\frac{\varepsilon}{\rho^2}(\rho_1 + \rho_{1\theta\theta}) + \mathcal{O}(\varepsilon^2)- P_s(\rho + \varepsilon\rho_{1}) \\
	&= P_s(\rho)- P_s(\rho + \varepsilon\rho_{1})- \gamma\frac{\varepsilon}{\rho^2}(\rho_1 + \rho_{1\theta\theta}) + \mathcal{O}(\varepsilon^2).
	\end{aligned}
	\end{equation*}
	
	We differentiate the above equation along $\Gamma_{\varepsilon}$ and get
	$$
	\|P-P_s\|_{C^{1+\alpha}(\Gamma_\varepsilon)} \le C|\varepsilon|\|\rho_{1}\|_{C^{3+\alpha}(\mathbb{R})}.
	$$
	The Schauder estimates imply
	\begin{equation*}
	\begin{aligned}
	\|P-P_s\|_{C^{1+\alpha}(\bar{\Omega}_\varepsilon)} &\le C\|M-M_s \|_{C^{\alpha}(\bar{\Omega}_\varepsilon)}  + C\|P-P_s\|_{C^{1+\alpha}(\Gamma_\varepsilon)}\\
	&\le C|\varepsilon|\|\rho_{1}\|_{C^{3+\alpha}(\mathbb{R})}.
	\end{aligned}
	\end{equation*}
	Due to the regularity of $M_s$ and $P_s$ and $\Gamma_\varepsilon\in C^{3+\alpha}$, we conclude the constant $C$ is independent of $\varepsilon$.
	
\end{proof}
Next, we proceed to rigorously establish \eqref{eqn:expansion}. Since both $M$ and $P$ are defined on $\Omega_\varepsilon$ but $M_1$ and $P_1$ are  defined on $\Omega$ only, we transform $M_1$ and $P_1$ to $\Omega_\varepsilon$ by Hanzawa transformation $H_\varepsilon$ \cite{hanzawa1981classical}:
$$
(r,\theta) = H_\varepsilon(r',\theta') = (r'+ \chi(r'-\rho)\varepsilon\rho_1,\theta')
$$
where
\begin{equation*}
\chi\in C^{\infty},\quad \chi(z) = \left\{
\begin{aligned}
&0\ \mathrm{if}\ |z| \ge \frac{3}{4}\delta_0\\
&1\ \mathrm{if}\ |z| < \frac{1}{4}\delta_0\\
\end{aligned}
\right. , \quad |\frac{d^k \chi}{dz^k}|\le\frac{C}{\delta^k_0}
\end{equation*}
and $\delta_0>0$ is small. Noticing that $H_\varepsilon$ maps $\Omega$ onto $\Omega_\varepsilon$ but keeps the annulus $\{r: \rho + \frac{3}{4}\delta_0 \le r \le R\}$ fixed, we set
\begin{equation}
\tilde{M}_1(r,\theta) = M_1(H^{-1}_\varepsilon(r,\theta)),\quad\tilde{P}_1(r,\theta) = P_1(H^{-1}_\varepsilon(r,\theta)).
\label{eqn:hanzawa}
\end{equation}
Then we establish the following estimates.
\begin{theorem}
	If $\rho_{1}\in C^{3+\alpha}(\mathbb{R})$, $(M,P)$ is the solution of \eqref{eqn:perturbed}, and $(\tilde{M}_1,\tilde{P}_1)$ is defined as \eqref{eqn:hanzawa}, then
	\begin{equation}
	\begin{aligned}
	&\|M - M_s - \varepsilon \tilde{M}_1 \|_{C^{3+\alpha}(\bar{\Omega}_\varepsilon)}\le C|\varepsilon|^2\|\rho_1\|_{C^{3+\alpha}(\mathbb{R})}\\
	&\|P - P_s - \varepsilon \tilde{P}_1 \|_{C^{1+\alpha}(\bar{\Omega}_\varepsilon)}\le C|\varepsilon|^2\|\rho_1\|_{C^{3+\alpha}(\mathbb{R})}
	\end{aligned}
	\end{equation}
\end{theorem}

\begin{proof}
	First, we compute the first and second derivatives of $\tilde{M}_1$ with respect to both $r$ and $\theta$:
	\begin{equation*}
	\begin{aligned}
	&\frac{\partial \tilde{M}_1}{\partial r} = \frac{\partial M_1}{\partial r'} \frac{\partial r'}{\partial r} ,\quad \frac{\partial \tilde{M}_1}{\partial \theta} = \frac{\partial M_1}{\partial r'} \frac{\partial r'}{\partial \theta}+ \frac{\partial M_1}{\partial \theta'},\\
	&\frac{\partial^2 \tilde{M}_1}{\partial r^2} = \frac{\partial^2 M_1}{\partial r'^2} (\frac{\partial r'}{\partial r})^2+\frac{\partial M_1}{\partial r'} \frac{\partial^2 r'}{\partial r^2},\\
	&\frac{\partial^2 \tilde{M}_1}{\partial \theta^2} =  \frac{\partial^2 M_1}{\partial \theta'^2}+2\frac{\partial^2 M_1}{\partial r'\partial \theta'} \frac{\partial r'}{\partial \theta}+\frac{\partial^2 M_1}{\partial r'^2} (\frac{\partial r'}{\partial \theta})^2+\frac{\partial M_1}{\partial r'} \frac{\partial^2 r'}{\partial \theta^2},\\
	\end{aligned}
	\end{equation*}
	where the derivatives of $r'$ is derived from the Hanzawa transformation. In fact, the first derivatives are
	\begin{equation*}
	\begin{aligned}
	&1 = \frac{\partial r'}{\partial r}+\varepsilon\rho_1\chi'(r'-\rho)\frac{\partial r'}{\partial r},\\
	&0 = \frac{\partial r'}{\partial \theta}+\varepsilon\rho_1\chi'(r'-\rho)\frac{\partial r'}{\partial \theta} + \varepsilon\chi(r'-\rho)\rho_{1\theta}
	\end{aligned}
	\end{equation*}
	thus
	\begin{equation*}
	\begin{aligned}
	\frac{\partial r'}{\partial r}&= \frac{1}{1+\varepsilon\rho_1\chi'(r'-\rho)}\hbox{~and~} \frac{\partial r'}{\partial \theta} = - \frac{\varepsilon\chi(r'-\rho)\rho_{1\theta}}{1+\varepsilon\rho_1\chi'(r'-\rho)}.
	\end{aligned}
	\end{equation*}
	Similarly, we obtain the second derivatives below
	\begin{equation*}
	\begin{aligned}
	\frac{\partial^2 r'}{\partial r^2}&= -\frac{\varepsilon\rho_1\chi''(r'-\rho)}{(1+\varepsilon\rho_1\chi'(r'-\rho))^2}\frac{\partial r'}{\partial r}=-\frac{\varepsilon\rho_1\chi''(r'-\rho)}{(1+\varepsilon\rho_1\chi'(r'-\rho))^3}, \\
	\frac{\partial^2 r'}{\partial \theta^2}  &= - \frac{\varepsilon\chi(r'-\rho)\rho_{1\theta\theta}}{1+\varepsilon\rho_1\chi'(r'-\rho)}+2\frac{\varepsilon^2\chi(r'-\rho)\chi'(r'-\rho)\rho^2_{1\theta}}{(1+\varepsilon\rho_1\chi'(r'-\rho))^2}\\
	&-\frac{(\chi(r'-\rho)\varepsilon\rho_{1\theta})^2\chi''(r'-\rho)\varepsilon\rho_1}{(1+\varepsilon\rho_1\chi'(r'-\rho))^3}.
	\end{aligned}
	\end{equation*}
	
	Next we consider the estimate of $\phi = M - M_s - \varepsilon \tilde{M}_1$ which satisfies:
	\begin{equation*}\left\{
	\begin{aligned}
	&\Delta \phi - \frac{H}{D}\phi =  \varepsilon^2\tilde{f} \quad\mathrm{in}\quad \Omega_\varepsilon\\
	&\phi = g \quad\mathrm{on}\quad \Gamma_\varepsilon\\
	&\frac{\partial\phi}{\partial r} + \phi = 0 \quad\mathrm{on}\quad \Gamma_1\\
	\end{aligned}
	\right.
	\end{equation*}	
	where $\tilde{f}$ depends on various terms of Hanzawa transform above and involves up to second order derivatives of $\rho_1$ and $M_1$. By applying the Schauder estimate to \eqref{eqn:linearized}, we know $M_1\in C^{3+\alpha}$ and
	$$
	\|\tilde{f}\|_{C^{1+\alpha}(\bar{\Omega}_\varepsilon)} \le C\|\rho_1\|_{C^{3+\alpha}(\mathbb{R})}.
	$$
	
	On the boundary $\Gamma_\varepsilon$, we have
	\begin{equation*}
	\begin{aligned}
	g  &= M(\rho+\varepsilon\rho_{1}) - M_s(\rho+\varepsilon\rho_{1}) - \varepsilon\tilde{M}_1(\rho+\varepsilon\rho_{1})\\
	&= 1 - M_s(\rho+\varepsilon\rho_{1}) + \varepsilon\rho_1\frac{\partial M_s(\rho) }{\partial r}\\
	&= M_s(\rho)- M_s(\rho+\varepsilon\rho_{1}) + \varepsilon\rho_1\frac{\partial M_s(\rho) }{\partial r}\\
	&= \mathcal{O}(\varepsilon^2)\rho_1.
	\end{aligned}
	\end{equation*}
	
	By the Schauder theory, we obtain
	$$
	\|M - M_s - \varepsilon \tilde{M}_1 \|_{C^{3+\alpha}(\bar{\Omega}_\varepsilon)}\le C|\varepsilon|^2\|\rho_1\|_{C^{3+\alpha}(\mathbb{R})}.
	$$
	Similarly,  we write the equation of $\psi = P - P_s - \varepsilon \tilde{P}_1$ as follows
	\begin{equation*}\left\{
	\begin{aligned}
	&-\Delta \psi = L\phi +  \varepsilon^2\tilde{k} \quad\mathrm{in}\quad \Omega_\varepsilon,\\
	&\psi = f \quad\mathrm{on}\quad \Gamma_\varepsilon,\\
	&\frac{\partial\psi}{\partial r} = 0 \quad\mathrm{on}\quad \Gamma_1,\\
	\end{aligned}
	\right.
	\end{equation*}	
	where $\tilde{k}$ is based on various term of Hanzawa transform above and follows
	$$
	\|\tilde{k}\|_{C^{1+\alpha}(\bar{\Omega}_\varepsilon)} \le C\|\rho_1\|_{C^{3+\alpha}(\mathbb{R})}.
	$$
	Since 	$$
	f = P(\rho+\varepsilon\rho_{1}) - P_s(\rho+\varepsilon\rho_{1}) - \varepsilon\tilde{P}_1(\rho+\varepsilon\rho_{1}),
	$$
	we have
	$$
	\|f \|_{C^{1+\alpha}(\mathbb{R})} \le C|\varepsilon|^2\|\rho_1\|_{C^{3+\alpha}(\mathbb{R})}.
	$$
	Therefore, by Schauder estimates, we conclude
	$$
	\|P - P_s - \varepsilon \tilde{P}_1 \|_{C^{1+\alpha}(\bar{\Omega}_\varepsilon)}\le C|\varepsilon|^2\|\rho_1\|_{C^{3+\alpha}(\mathbb{R})}.
	$$		
\end{proof}

\subsection{Bifurcation analysis}
\label{subsec:bifur}
We consider the nonlinear function $F$ defined in (\ref{eqn:nonlinear}) by expanding $\frac{\partial P}{\partial r}$ on $\Gamma_\varepsilon$, namely,
\begin{equation}
F(\rho_1,L)=	\frac{\partial P}{\partial r}\vline_{\Gamma_\varepsilon} = \varepsilon (\frac{\partial P_1}{\partial r}(\rho) +\frac{\partial^2 P_s(\rho) }{\partial r^2}\rho_1) + \mathcal{O}(|\varepsilon|^2).
\label{eqn:frechet}
\end{equation}
Thus $F$ maps $(\rho_{1},L)$ from $X^{l+3+\alpha}$ to $X^{l+\alpha} $ and the same argument is bounded for any $l\ge0$ \cite{gilbarg2015elliptic}. Furthermore, $F$ is Fr$\mathrm{\acute{e}}$chet differentiable and the Fr$\mathrm{\acute{e}}$chet  derivative at $(0,L)$ is given by
\begin{equation}
[\frac{\partial F}{\partial \rho_1}(0,L)]\cos(n\theta)=\frac{\partial P_1}{\partial r}(\rho) +\frac{\partial^2 P_s(\rho) }{\partial r^2}\rho_1.
\end{equation}
Then the bifurcation condition  becomes
\begin{equation}
\frac{\partial P_1}{\partial r}(\rho) +\frac{\partial^2 P_s(\rho) }{\partial r^2}\rho_1 = 0.
\end{equation}
Since	
\begin{equation}
\begin{aligned}
\frac{\partial P_1}{\partial r}(\rho) = \cos(n\theta)[-\frac{LD}{H}\frac{\partial Q_n}{\partial r}(\rho) +n \tilde C_3 \rho^{n-1} -n\tilde C_4 \rho^{-n-1} ]
\end{aligned}
\end{equation}
and	
\begin{equation}
\begin{aligned}
\frac{\partial^2 P_s(\rho)}{\partial r^2} &=T(L)-LM_s(\rho) = T(L)-L,
\end{aligned}
\end{equation}
we obtain	
\begin{equation}
\begin{aligned}
F(L) =& T-L-\frac{LD}{H}\frac{\partial Q_n}{\partial r}(\rho) +n \tilde C_3 \rho^{n-1} -n\tilde C_4 \rho^{-n-1} \\
=& -\frac{DL}{H}\frac{2}{R^2-\rho^2}(RM_s(R)+\rho M'_s(\rho))-L-\frac{LD}{H}Q'_n(\rho) \\
&-2\frac{L\frac{D}{H}Q_n(R)R^{n+1}}{(\rho^{2n}+R^{2n})}\rho^{n-1}+\frac{\gamma(n^3-n)(\rho^{2n}-R^{2n})}{\rho^3(\rho^{2n}+R^{2n})}\\
&+n\frac{LD}{H}\frac{\rho^{2n}-R^{2n}}{\rho(\rho^{2n}+R^{2n})}Q_n(\rho)\\
=&0.
\end{aligned}
\end{equation}
Therefore, the formula of $L_n$ for bifurcation points is
\begin{equation}
L_n = \frac{C_1(n,\rho,R)}{C_2(n,\rho,R)}
\label{eqn:bifur}
\end{equation}
where
$$
C_1(n,\rho,R) = \frac{\gamma(n^3-n)(R^{2n}-\rho^{2n})}{\rho^3(\rho^{2n}+R^{2n})}
$$
and
\begin{equation*}
\begin{aligned}
C_2(n,\rho,R) =& -\frac{D}{H}\frac{2}{R^2-\rho^2}(RM_s(R)+\rho M'_s(\rho))-\frac{D}{H}Q'_n(\rho)\\
&-\frac{D}{H}\frac{2R^{n+1}\rho^{n}Q_n(R)+n(R^{2n}-\rho^{2n})Q_n(\rho)}{\rho(\rho^{2n}+R^{2n})}-1.
\end{aligned}
\end{equation*}
It is clear that $C_1$ is increasing with respect to $n$ while the monotonicity of $C_2(n,\rho,R)$ is  summarized in the following lemma.

\begin{lemma}
	For given $R>0$, $\rho$ is in a neighbor of $R$, namely, $\rho = R-\varepsilon$ for a small $\varepsilon$, $C_2(n,\rho,R) > 0$ is decreasing with respect to $n$ .
	\label{lemma:bifur_mono}
\end{lemma}
\begin{proof}
	We rewrite $C_2(n,\rho,R)$ as
	$$		C_2(n,\rho,R) = -\frac{D}{H}\frac{2}{R^2-\rho^2}(RM_s(R)+\rho M'_s(\rho))-1+\frac{D}{H}f(n),$$
	where		$$
	f(n) = -Q'_n(\rho)-\frac{2R^{n+1}\rho^{n}}{\rho(\rho^{2n}+R^{2n})}Q_n(R)-\frac{n(R^{2n}-\rho^{2n})}{\rho(\rho^{2n}+R^{2n})}Q_n(\rho).
	$$
	Since  $\rho = R-\varepsilon$, we have
	$$
	f(n) = \frac{H}{D}Q_n(R) \varepsilon + \mathcal{O}(\varepsilon^2).
	$$
	By letting $F = \frac{d Q_n}{dn}$, we obtain
	\begin{equation*}
	\left\{
	\begin{aligned}
	-\Delta F + (\frac{n^2}{r^2}+\frac{H}{D}) F &= -\frac{2n}{r^2}Q_n,\\
	F(\rho) &= 0,\\
	F(R) &= \frac{d Q_n}{dn}(R),
	\end{aligned}
	\right.
	\end{equation*}
	and 		$$
	\frac{dF}{dr}(R) = \frac{d}{dr} \frac{d Q_n}{dn}(R)  = \frac{d}{dn} \frac{d Q_n}{dr}(R) = - \frac{dQ_n}{dn}(R).
	$$
	If $\frac{d Q_n}{dn}(R) \ge 0$, then we have $\frac{dF}{dr}(R)<0$. On the other hand,  by the maximum principle, we have $ \frac{dF}{dr}(R) >0$, which leads to a contradiction. Thus $\frac{d Q_n}{dn}(R) < 0$,  we have $C_2(n)$ decreases with respect to $n$ and $F(r) \le 0$ for all $\rho \le r \le R$.	
	Moreover, since $	\frac{dF}{dr}(\rho) <0 $, we have $\frac{d Q'_n(\rho)}{dn} < 0$.
	
	Next we prove that $C_2(n,\rho,R) > 0$ when $\varepsilon$ is small and expand $C_2(n,\rho,R)$  in terms of $\varepsilon$
	\begin{equation*}
	\begin{aligned}
	C_2(n,\rho,R)
	=& (\frac{1}{2}M'_s(R) + Q_n(R))\varepsilon +  \mathcal{O}(\varepsilon^2)\\
	=& (\frac{1}{2}M'_s(R) + Q_n(\rho)+\varepsilon Q'_n(\rho))\varepsilon +  \mathcal{O}(\varepsilon^2)\\
	=&(\frac{1}{2}M'_s(R) -M'_s(\rho))\varepsilon +  \mathcal{O}(\varepsilon^2)\\
	=&-\frac{1}{2}M'_s(R)\varepsilon +  \mathcal{O}(\varepsilon^2).
	\end{aligned}
	\end{equation*}
	Since $M'_s(R) < 0$, we have $C_2(n,\rho,R) > 0$ for a small $\varepsilon$.
	
\end{proof}

Then we prove that $L_n$ in (\ref{eqn:bifur}) is a bifurcation point by verifying the following Crandall-Rabinowitz theorem \cite{crandall1971bifurcation}.

\begin{theorem}	
	Let  $X,Y$  be real Banach spaces and $F(x,\mu)$ a $C^p$ map, $p\ge3$, of a neighborhood $(0,\mu_0)$ in $X\times\mathbb{R}$ into Y. Suppose
	\begin{enumerate}
		\item $F(0,\mu)=0$ for all $\mu$ in a neighborhood of $\mu_0$,
		\item $Ker F_x(0,\mu_0)$ is one dimensional space, spanned by $x_0$,
		\item $Im F_x(0,\mu_0) = Y_1$ has codimension 1,
		\item $F_{\mu x}(0,\mu_0)\notin Y_1$.
	\end{enumerate}
	
	Then $(0,\mu_0)$ is a bifurcation point of the equation $F(x,\mu) = 0$ in the following sense: In a neighborhood of  $(0,\mu_0)$ the set of solutions of $F(x,\mu) = 0$ consists
	of two $C^{p-2}$ smooth curves $\mathcal{C}_1$ and $\mathcal{C}_2$ which intersect only at the point  $(0,\mu_0)$;  $\mathcal{C}_1$  is the curve $(0,\mu_0)$ and $\mathcal{C}_2$
	can be parameterized as follows:
	$$
	\mathcal{C}_2: (x(\varepsilon),\mu(\varepsilon)), |\varepsilon| \ small,\ (x(0),\mu(0)) = (0,\mu_0), x'(0)=x_0.
	$$
	\label{thm:crandall}
\end{theorem}

\begin{proof}[Verification]
	We choose the Banach spaces $X = X_1^{3+\alpha}$ and $Y = X_1^{\alpha}$ then have
	$$
	[F_{\rho_1}(0,L)]\cos(n\theta)=(C_1(n,\rho,R)-LC_2(n,\rho,R))\cos(n\theta).
	$$
	Thus the kernel space satisfies
	\begin{equation}
	\mathrm{ker}[F_{\rho_1}(0,L)] = \mathrm{span}\{\cos(n\theta)\}\quad \mathrm{if}\ L = L_n
	\end{equation}
	and
	\begin{equation}
	\mathrm{ker}[F_{\rho_1}(0,L)] = 0 \quad \mathrm{if}\ L \ne L_1,L_2,\cdots
	\end{equation}
	which implies that $\mathrm{dim}(\mathrm{ker}[F_{\rho_1}(0,L)] ) = 1$. Moreover, since that $\mathrm{Im}[F_{\rho_1}(0,L_n)] \oplus \{\cos(n\theta)\}$ is the whole space, we have $\mathrm{codim}(\mathrm{Im}[F_{\rho_1}(0,L_n)]) = 1$.
	Finally, by differentiating with respect to $L$, we obtain
	$$
	[F_{\rho_1L}(0,L)]\cos(n\theta) =-C_2(n,\rho,R))\cos(n\theta) \notin\mathrm{Im}[F_{\rho_1}(0,L_n)] .
	$$	
	Thus all the assumptions in the Crandall-Rabinowitz theorem are satisfied.
\end{proof}

\subsection{Linear Stability}
We consider the linear stability via linearizing  the free boundary $\Gamma_\varepsilon(t)$, $M(r,\theta,t)$, and $P(r,\theta,t)$ as follows:
\begin{equation}
\begin{aligned}
&\Gamma_\varepsilon: r = \rho_0(\theta) +\varepsilon\rho_1(\theta,t)+\mathcal{O}(\varepsilon^2),\\
&M(r,\theta,t) = M_0(r,\theta) +\varepsilon M_1(r,\theta,t) +\mathcal{O}(\varepsilon^2),\\
&P(r,\theta,t) = P_0(r,\theta) +\varepsilon P_1(r,\theta,t) +\mathcal{O}(\varepsilon^2).
\end{aligned}
\end{equation}
The linearization of the normal direction of $\Gamma_\varepsilon$ is
\begin{equation}
\label{eqn:normal}
\begin{aligned}
\vec{n}|_{r= \rho_0 +\varepsilon\rho_1} &= -\frac{\vec{e}_r-\frac{1}{r}(\rho_{ 0\theta}+\varepsilon\rho_{1\theta})\vec{e}_\theta}{\sqrt{1+\frac{1}{r^2}(\rho_{0 \theta}+\varepsilon\rho_{1\theta})^2}}\\
&=-\frac{(\rho_0 +\varepsilon\rho_1)\vec{e}_r-(\rho_{ 0\theta}+\varepsilon\rho_{1\theta})\vec{e}_\theta}{\sqrt{(\rho_0 +\varepsilon\rho_1)^2+(\rho_{ 0\theta}+\varepsilon\rho_{1\theta})^2}}\\
&=\vec{n}_0 + \varepsilon\vec{n}_1 + \mathcal{O}(\varepsilon^2),
\end{aligned}
\end{equation}
where

\begin{equation*}
\begin{aligned}
\vec{n}_0 & = - \frac{\rho_0}{(\rho_0^2+\rho_{0\theta}^2)^{\frac{1}{2}}}\vec{e}_r+
\frac{\rho_{0\theta}}{(\rho_0^2+\rho_{0\theta}^2)^{\frac{1}{2}}}\vec{e}_\theta
\end{aligned}
\end{equation*}

and

\begin{equation*}
\begin{aligned}
\vec{n}_1 &= -[\frac{\rho_1\vec{e}_r-\rho_{1\theta}\vec{e}_\theta}{(\rho_0^2+\rho_{ 0\theta}^2)^{\frac{1}{2}}}-\frac{(\rho_0\vec{e}_r-\rho_{0\theta}\vec{e}_\theta)(\rho_0\rho_1+\rho_{0\theta}\rho_{1 \theta})}{(\rho_0^2+\rho_{0\theta}^2)^{\frac{3}{2}}}]\\
& = -\frac{\rho_1\rho_{0\theta}^2-\rho_0\rho_{ 0\theta}\rho_{1\theta}}{(\rho_0^2+\rho_{0\theta}^2)^{\frac{3}{2}}}\vec{e}_r  - \frac{\rho_0\rho_{1}\rho_{ 0\theta}-\rho_0^2\rho_{1 \theta}}{(\rho_0^2+\rho_{0 \theta}^2)^{\frac{3}{2}}}\vec{e}_\theta.
\end{aligned}
\end{equation*}

Then we linearize $\frac{\partial P}{\partial n}$ on the free boundary

\begin{equation*}
\begin{aligned}
&\frac{\partial P}{\partial n}|_{r= \rho_0 +\varepsilon\rho_1} \\
=&\{(P_{0r} +\varepsilon P_{1r})\vec{e}_r + \frac{1}{r}(P_{0\theta} +\varepsilon P_{1\theta})\vec{e}_\theta\}_{r= \rho_0 +\varepsilon\rho_1}\cdot(\vec{n}_0 + \varepsilon\vec{n}_1) \\
=& (P_{0r}\vec{e}_r +\frac{P_{0\theta} }{\rho_0}\vec{e}_\theta)\cdot \vec{n}_0 + \varepsilon(P_{0r}\vec{e}_r +\frac{P_{0\theta} }{\rho_0}\vec{e}_\theta)\cdot \vec{n}_1 \\
& + \varepsilon[(\rho_1P_{0rr}+P_{1r})\vec{e}_r + \frac{\rho_1\rho_0P_{0r\theta}+ P_{1\theta}\rho_0-P_{0\theta}\rho_1}{\rho_0^2}\vec{e}_\theta ] \cdot \vec{n}_0.
\end{aligned}
\end{equation*}
Since $(P_{0r}\vec{e}_r +\frac{P_{0\theta} }{\rho_0}\vec{e}_\theta)\cdot \vec{n}_0 = \frac{\partial  P_0}{\partial n_0} = 0$, we have	
\begin{equation*}
\begin{aligned}
&  [(\rho_1P_{0rr}+P_{1r})\vec{e}_r + \frac{\rho_1\rho_0P_{0r\theta}+ P_{1\theta}\rho_0-P_{0\theta}\rho_1}{\rho_0^2}\vec{e}_\theta ] \cdot \vec{n}_0\\
+&(P_{0r}\vec{e}_r +\frac{P_{0\theta} }{\rho_0}\vec{e}_\theta)\cdot \vec{n}_1\\
=& -\frac{P_{0r}(\rho_1\rho_{0 \theta}^2-\rho_0\rho_{0 \theta}\rho_{1\theta})}{(\rho_0^2+\rho_{0 \theta}^2)^{\frac{3}{2}}}- \frac{P_{0\theta} (\rho_0\rho_{1}\rho_{0 \theta}-\rho_0^2\rho_{1 \theta})}{\rho_0(\rho_0^2+\rho_{0 \theta}^2)^{\frac{3}{2}}}\\
&- \frac{\rho_0(\rho_1P_{0rr}+P_{1r})}{(\rho_0^2+\rho_{0 \theta}^2)^{\frac{1}{2}}}
+ \frac{\rho_{0\theta}(\rho_1\rho_0P_{0r\theta}+ P_{1\theta}\rho_0-P_{0\theta}\rho_1)}{\rho_0^2(\rho_0^2+\rho_{0 \theta}^2)^{\frac{1}{2}}}
\end{aligned}
\end{equation*}

On the other hand, the velocity of $\Gamma_\varepsilon$ along the normal direction is	
\begin{equation*}
\begin{aligned}
v_n & = \frac{-(\rho_0+\varepsilon \rho_1)_t}{\sqrt{1+\frac{1}{(\rho_0 +\varepsilon\rho_1)^2}(\rho_{0 \theta}+\varepsilon\rho_{1\theta})^2}} = -\frac{\varepsilon\rho_{1t}}{\sqrt{1+\frac{1}{\rho_0^2}\rho_{0 \theta}^2}} + \mathcal{O}(\varepsilon^2).
\end{aligned}
\end{equation*}

Since $\frac{\partial P}{\partial n} = -v_n$ on $\Gamma_\varepsilon$, we obtain the following equation for $\rho_1$
\begin{equation*}
\begin{aligned}
\rho_{1t}=& -[\frac{P_{0r}\rho_{0 \theta}^2+P_{0\theta}\rho_{0 \theta}}{(\rho_0^2+\rho_{0 \theta}^2)\rho_0}+P_{0rr}- \frac{\rho_{0\theta}P_{0r\theta}}{\rho_0^2}+ \frac{\rho_{0\theta}P_{0\theta}}{\rho_0^3}]\rho_1\\
& +\frac{P_{0r}\rho_{0 \theta}+P_{0\theta}}{\rho_0^2+\rho_{0 \theta}^2}\rho_{1 \theta} - [P_{1r} - \frac{P_{1\theta}\rho_{0 \theta}}{\rho_0^2}].
\end{aligned}
\end{equation*}
Then the linearized system is
\begin{equation}
\label{eqn:linear_time}
\left\{
\begin{aligned}
&\frac{\partial M_1}{\partial t}-D\Delta M_1 = -HM_1\\
&\Delta P_1 = -LM_1\\
&\frac{\partial M_1}{\partial r}(R) = -M_1(R)\\
&M_{1}(\rho_0) = -\rho_1 M_{0r}(\rho_0)\\
&P_1(\rho_0) =\gamma \kappa_1-\rho_1P_{0r}\\
&\frac{\partial P_1}{\partial r}(R)  = 0\\
&\rho_{1t}|_{r=\rho_0}= [\frac{\rho_{0\theta}P_{0r\theta}}{\rho_0^2}-\frac{P_{0r}\rho_{0 \theta}^2+P_{0\theta}\rho_{0 \theta}}{(\rho_0^2+\rho_{0 \theta}^2)\rho_0}-P_{0rr}- \frac{\rho_{0\theta}P_{0\theta}}{\rho_0^3}]\rho_1\\
&+\frac{P_{0r}\rho_{0 \theta}+P_{0\theta}}{\rho_0^2+\rho_{0 \theta}^2}\rho_{1 \theta} - [P_{1r} - \frac{P_{1\theta}\rho_{0 \theta}}{\rho_0^2}]
\end{aligned}
\right.
\end{equation}

Next we summarize the linear instability of the radially symmetric steady-state solutions in the following theorem.
\begin{theorem}
	For any given $L > 0$,  the corresponding radially symmetric steady-state solution $(M_s(r),P_s(r),\rho_0)$ is linearly unstable. In fact, there exists initial conditions defined by
	\begin{equation}
	\begin{aligned}
	&\rho_1(0) = \cos(n\theta)\\
	&M_1(r,0) = u(r)\cos(n\theta)\\
	&P_1(r,0) = w(r)\cos(n\theta)\\
	\end{aligned}
	\end{equation}
	such that  $\rho_1(t)\to\infty$.
	\label{Thm:stability}
\end{theorem}

\begin{proof}
	We consider the solution $(\rho_1(t),M_1(r,t),P_1(r,t))$ with the following form
	\begin{equation}
	\left\{
	\begin{aligned}
	&\rho_1(t) = e^{at}\cos(n\theta),\\
	&M_1(r,t) = e^{at}u(r)\cos(n\theta),\\
	&P_1(r,t) = e^{at}w(r)\cos(n\theta).\\
	\end{aligned}
	\right.
	\label{eqn:unstable_perturb}
	\end{equation}
	Then the linearized system (\ref{eqn:linear_time}) is written as
	\begin{equation}
	\left\{
	\begin{aligned}
	&au(r)-D(\Delta u(r)-\frac{n^2}{r^2}u(r)) = -Hu(r)\\
	&\frac{\partial u}{\partial r}(R) = -u(R)\\
	&u(\rho_0) = - M'_{s}(\rho_0)\\
	&\Delta w(r)-\frac{n^2}{r^2}w(r) = -Lu(r)\\
	&w(\rho_0) =\frac{\gamma}{\rho_0^2}(n^2-1)\\
	&\frac{\partial w}{\partial r}(R)  = 0\\
	&a= -P''_{s}- w_{r}(\rho_0)
	\end{aligned}
	\right.
	\label{eqn:stability}
	\end{equation}
	
	By repeating the process in Section~\ref{subsec:linearized} (from (\ref{eqn:Qn}) to (\ref{eqn:C4})), we conclude that $a$  satisfies the following equation
	\begin{equation}
	\begin{aligned}
	a &= L - T(L)  + \frac{LD}{H+a}u_r(\rho_0)\\
	&+\frac{DL}{H+a}\frac{2R^{n+1}\rho_0^{n}u(R)+n(R^{2n}-\rho_0^{2n})u(\rho_0)}{\rho_0(\rho_0^{2n}+R^{2n})},
	\end{aligned}
	\end{equation}	where\begin{equation*}
	u(r) =-M'_{s}(\rho_0) (\hat{C}_1I_n(\hat{z}_r)+\hat{C}_2K_n(\hat{z}_r)),~ \hat{z}_r = \sqrt{\frac{H+a}{D}r},
	\end{equation*}
	\begin{equation*}
	\left\{
	\begin{aligned}
	&\hat{C}_1 =\frac{1}{I_n(\hat{z}_\rho)+KK_n(\hat{z}_\rho)}, \\
	&\hat{C}_2= \frac{K}{I_n(\hat{z}_\rho)+KK_n(\hat{z}_\rho)},\\
	\end{aligned}
	\right.
	\end{equation*}
	and
	$$
	K = -\frac{I_n(\hat{z}_R)+I'_n(\hat{z}_R)}{K_n(\hat{z}_R)+K'_n(\hat{z}_R)}.
	$$
	We consider a nonlinear function $h(a,n,L)$ defined as
	\begin{equation*}
	\begin{aligned}
	h(a,n,L)=&L - T(L)  -a + \frac{LD}{H+a}u_r(\rho_0)\\
	&+\frac{DL}{H+a}\frac{2R^{n+1}\rho_0^{n}u(R)+n(R^{2n}-\rho_0^{2n})u(\rho_0)}{\rho_0(\rho_0^{2n}+R^{2n})}.
	\end{aligned}
	\end{equation*}
	For $n\ge 0$, we have
	$$
	h(\infty,n,L)  \to -\infty\hbox{~and~}	h(0,n,L) = C_1(n,\rho_0,R) - C_2(n,\rho_0,R) L.
	$$
	According to (\ref{eqn:bifur}), $L_n$ is monotonically increasing with respect to $n$. Thus for any given $L$, there exists $n^*$ such that  $L < L_n*$ which implies that $h(0,n^*,L) = C_1 - C_2 L>  C_1 - C_2 L_{n^*}=0$. Therefore there must be at least one positive root of $h(a,n^*,L) = 0$. By \eqref{eqn:unstable_perturb}, we have $\rho_1\to\infty$.
	
\end{proof}

\begin{remark}
	In the proof, we have $n^*\ge2$. In fact, when $n = 0,1$, for $L \ge 0$, we have
	$$
	h(0,n,L) = -C_2(n,\rho_0,R)L \le 0.
	$$
	Moreover, when $\rho_0$ is in a neighborhood of $R$, say $\rho_0 = R - \varepsilon$, we expand $h(a,n,L)$ at $r=R$:
	\begin{equation*}
	h(a,n,L)= -a + L - T(L)-\varepsilon Lu(R) + \mathcal{O}(\varepsilon^2)
	\end{equation*}
	which decreases with respect to $a$ when $\varepsilon$ small. Thus $h(a,n,L) = 0$ does not have a positive solution for $n=0,1$ and fixed $L \ge 0$.
	
\end{remark}

\section{Numerical Results}	\label{sec:numerical}	
In this section, we employ numerical simulations to verify our theoretical results. Since the Laplacian operator in the 2D polar coordinate is defined as
$$
\Delta = \frac{\partial^2}{\partial r^2} + \frac{1}{r}\frac{\partial}{\partial r} + \frac{1}{r^2}\frac{\partial^2}{\partial \theta^2},
$$
we use the uniform grid points on the $\theta$ direction with a stepsize $\Delta\theta=\frac{2\pi}{m}$, namely, $\theta_j=j\Delta\theta$, $j=0,\dots,m-1$ where $m$ is the number of grid points on the $\theta$ direction. For the radius on each $\theta_j$ direction, we use the uniform grid points with a stepsize $h_j$, namely, $r_{i,j}=\rho_j+ih_j$ with $h_j=\frac{R-\rho_j}{n}$ and $i=0,1,\dots,n$, where $n$ is the number of grid points on each radius. Then we use the central difference scheme to approximate $\frac{\partial^2}{\partial r^2}G(r_{i,j},\theta_j)$ and $\frac{1}{r}\frac{\partial}{\partial r}G(r_{i,j},\theta_j)$, namely,
\begin{equation*}
\frac{\partial^2}{\partial r^2}G(r_{i,j},\theta_j)=\frac{G(r_{i+1,j},\theta_j)+G(r_{i-1,j},\theta_j)-2G(r_{i,j},\theta_j)}{h_j^2}
\end{equation*}
and
\begin{equation*}
\frac{\partial}{\partial r}G(r_{i,j},\theta_j)=\frac{G(r_{i+1,j},\theta_j)-G(r_{i-1,j},\theta_j)}{2h_j}.
\end{equation*}
Moreover, we use nine points to approximate  $\frac{\partial^2}{\partial \theta^2}$ such that the scheme has the second-order accuracy even for the non-radially symmetric case. The scheme is derived based on the Taylor expansion and shown in (\ref{ddtheta}) in the Appendix. This numerical scheme will recover the central difference scheme when the system reduces to the radially symmetric case.

\subsection{Convergence Test}
First we perform a  convergence order test of our numerical scheme for the radially symmetric steady-state solution which has analytic formulas shown in (\ref{sol:M_S}) and (\ref{sol:P_S}). The numerical error is defined as $Err(h,\Delta \theta)=\|(M_h,P_h)-(M_S,P_S)\|_\infty$ where $(M_h,P_h)$ is the numerical solution and $(M_S,P_S)$ is the analytic solution.
Here we choose $L = H=3$, $D=1$, $\gamma = 2$, $R = 2$ and $\rho=1.6$ and show the numerical error in Table \ref{tab:convergence} which demonstrates the second order of convergence. 	

	\begin{table}[ht!]
		    \centering
		\caption{Numerical errors and the convergence order for different grid points.  }
		\label{tab:convergence}
			\begin{tabular}{|c|c|c|}\hline
				$(h,\Delta \theta)$ & $Err(h,\Delta \theta)$ & order of convergence \\
				\hline
				$ (0.2,  \frac{\pi}{10})$	& 0.0192 & - \\
				\hline
				$ (0.1,  \frac{\pi}{20})$& 0.0044 & 2.1299\\
				\hline
				$ (0.05,  \frac{\pi}{40})$	& 0.0011 & 2.0303\\
				\hline
				$ (0.025,  \frac{\pi}{80})$&2.6668e-04 & 2.0075\\
				\hline
			\end{tabular}

	\end{table}

%
Next, we test the convergence of the numerical scheme on computing bifurcation points. Numerically we use the adaptive homotopy method \cite{hao2020adaptive} to compute bifurcation points. More specifically,  starting with a radially symmetric steady-state solution, we track along the radially symmetric solution path to $L$ and monitor the smallest eigenvalue of the nonlinear system. When the norm of the smallest eigenvalue is less than a  tolerance, e.g, $10^{-4}$ in our simulation, we obtain a numerical bifurcation point denote as $\tilde{L}$. The theoretical value of bifurcation point, $L_n$, is computed by \eqref{eqn:bifur} for any given  $n$.  Then we compute the numerical error of bifurcation points for $n = 2,3,4$ with different stepsize $h$ and $\Delta \theta$ shown in Table \ref{tab:accuracy}. It is clearly shown that the numerical error gets smaller when the stepsize gets smaller which demonstrates the convergence.

\begin{table}[ht]
	\centering
	\caption{The numerical error of bifurcation points $|\tilde{L}-L_n|$ for different $n$ and stepsize. }
	\label{tab:accuracy}
		\begin{tabular}{|c|c|c|c|}\hline
			$(h,\Delta \theta)$&$n=2$ & $n=3$ & $n=4$ \\
			\hline
			$ (0.2,  \frac{\pi}{10})$ & 0.7005 & 3.3115 & 1.9387 \\
			\hline
			$ (0.1,  \frac{\pi}{20})$& 0.1465 & 0.6722 & 0.1152 \\
			\hline
			$ (0.05,  \frac{\pi}{40})$& 0.0355 & 0.1586 & 0.0541  \\
			\hline
			$ (0.025,  \frac{\pi}{80})$&0.0106 & 0.0305 & 0.0122 \\
			\hline
		\end{tabular}
\end{table}
\subsection{The bifurcation structure and non-radially symmetric solutions}

We numerically explore the local bifurcation structure and non-radially symmetric steady-state solutions by using the tangent cone algorithm \cite{hao2020adaptive}.
The local bifurcation structure is shown in Fig.~\ref{fig:b2} for $n=2,3,4$. The $y$ axis is a projection function defined as  $\mathcal{P}(\rho(\theta))=(\rho_{max}-\rho_{min})(\theta_{max}-\theta_{min})$ for any given $\rho(\theta)$ which quantifies the change of the free boundary. For the radially symmetric branch, we have $\mathcal{P}(\rho(\theta))=0$; for the non-radially symmetric branch,  we have different local structures shown in Fig.~\ref{fig:b2} for different $n$. Moreover, the non-radially symmetric solutions in Fig.~\ref{fig:b2} are consistent with the perturbation $ \cos(n\theta)$ in section \ref{subsec:linearized}. The color of non-radially symmetric solutions stands for the value of $M$ in the domain.

\begin{figure}[htbp]
	\centering
	\includegraphics[width=0.5\linewidth]{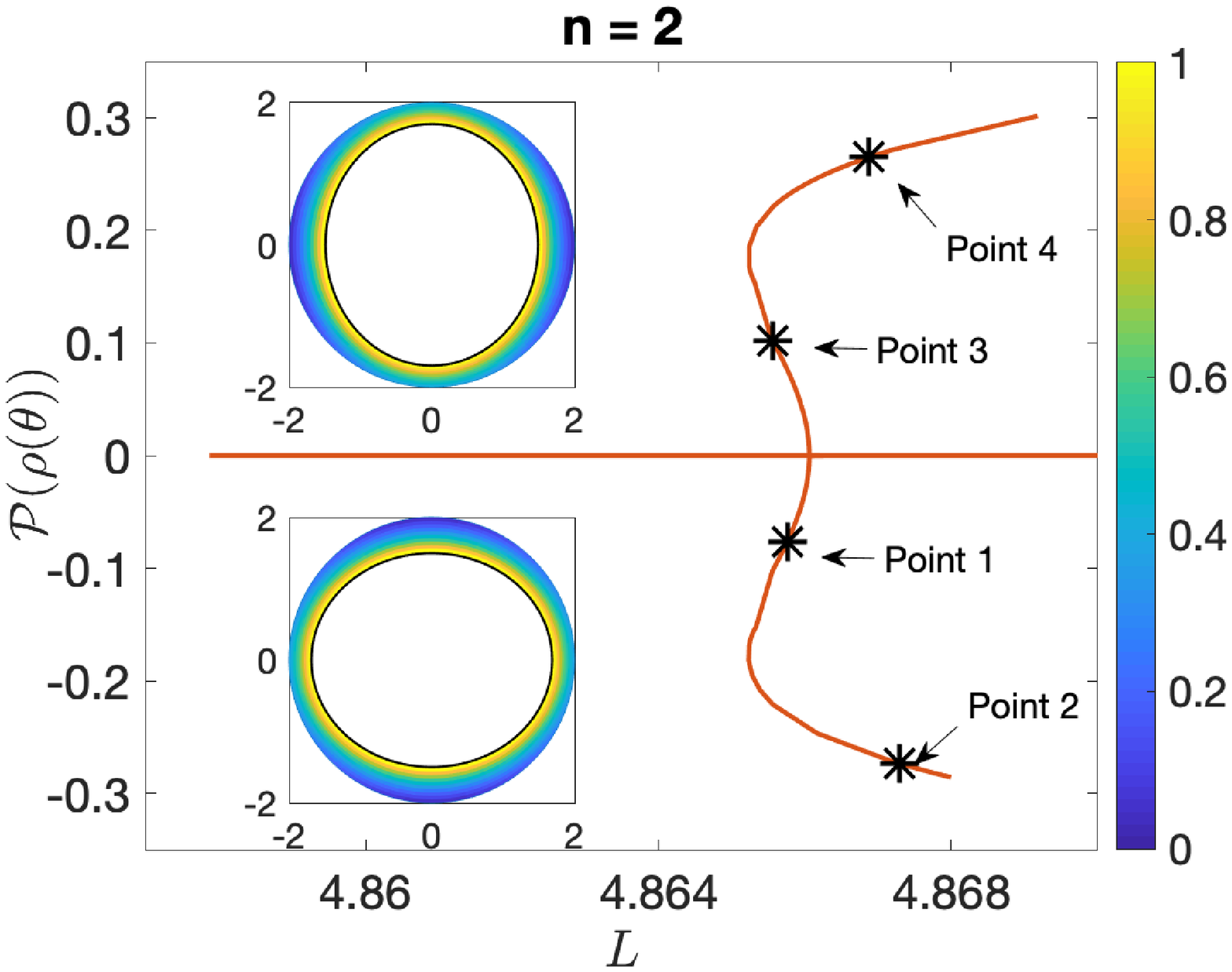}
	\includegraphics[width=0.5\linewidth]{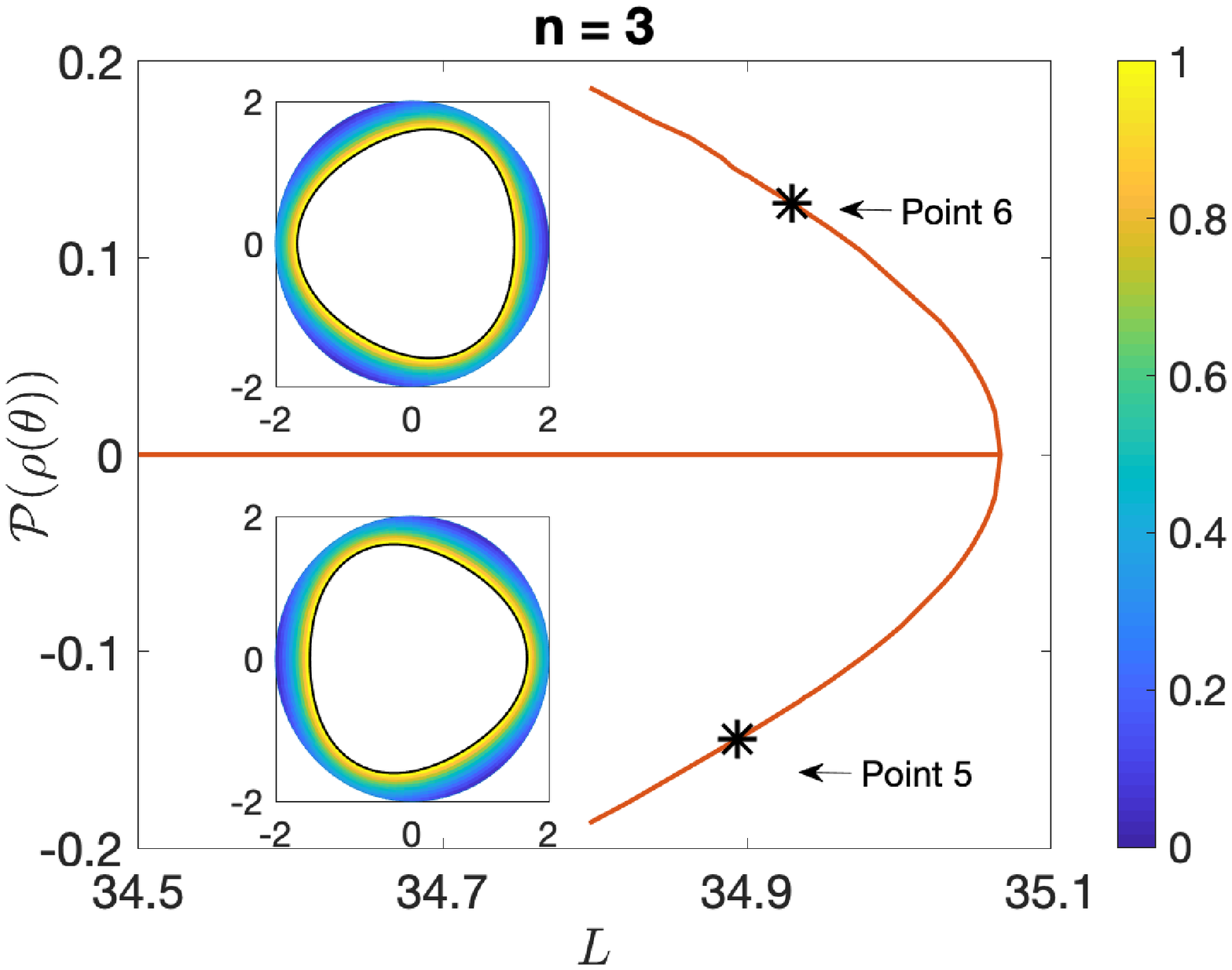}
	\includegraphics[width=0.5\linewidth]{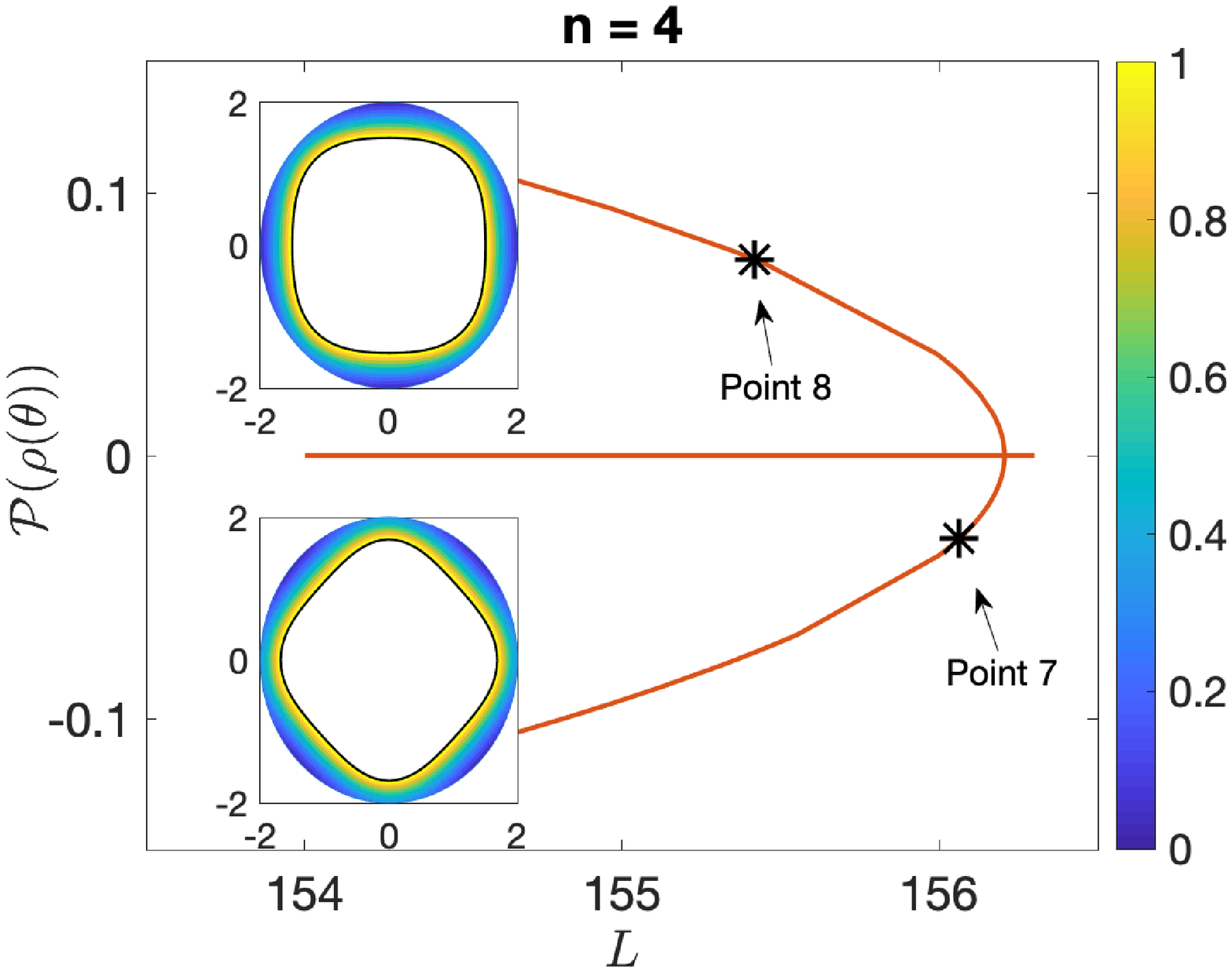}
	\caption{The bifurcation structure and non-radially symmetric steady-state solutions for $n=2,3,4$.}
	\label{fig:b2}
\end{figure}

\subsection{Linear Stability}
First, we verify the conclusion of Theorem~\ref{Thm:stability} via checking the linear stability of radially symmetric solutions for different values of $L$. More specifically, we check the real part of the largest eigenvalue, $real(\lambda_{max})$, and list in Table~\ref{tab:different_L} for different $L$. It shows that all the radially symmetric solutions are linearly unstable due to the positive largest eigenvalue.	
\begin{table}[ht]
	\centering
	\caption{\label{tab:different_L}The real of the largest eigenvalues of radially symmetric solutions v.s. $L$.  }
		\begin{tabular}{|c|c|}\hline
			$L$ & $real(\lambda_{max})$\\
			\hline
			$L = 0$ & 1.7050$\times 10^{4}$ \\
			\hline
			$L = 3$ & 1.7082 $\times 10^{4}$\\
			\hline
			$L = 15$ & 1.7085 $\times 10^{4}$\\
			\hline
			$L = 120$  & 1.7108 $\times 10^{4}$\\
			\hline
			$L = 200$  & 1.7126$\times 10^{4}$\\
			\hline
		\end{tabular}
\end{table}

Second, we check the linear stability of radially symmetric solution with radially symmetric perturbations which is the case of $n=0$ in {\bf Remark 1}. It can be seen from Table~\ref{tab:symm_L} that, under radially symmetric perturbations, the radially symmetric solutions are linearly stable when $L$ is small which is consistent with {\bf Remark 1}. As $L$ becomes large, there could be unstable coupled perturbations that can not be written in the form of separation of variables in (\ref{eqn:unstable_perturb}). Therefore, for large $L$, radially symmetric solutions become unstable even with radially symmetric perturbations shown in Table~\ref{tab:symm_L}.
\begin{table}[ht!]
	\centering
	\caption{\label{tab:symm_L}The real part of the largest eigenvalues of radially symmetric solutions v.s. $L$ under radially symmetric perturbations.  }
		\begin{tabular}{|c|c|}\hline
			$L$ & $real(\lambda_{max})$\\
			\hline
			5 & -1.7489\\
			\hline
			10 & -4.1792\\
			\hline
			50 & -7.3511 \\
			\hline
			100 & -3.5551\\
			\hline
			140 & -0.3859\\
			\hline
			150 & 0.4233\\
			\hline
			200 & 4.5619\\
			\hline
		\end{tabular}
\end{table}

Last, we test the linear stability on the non-radially symmetric branches by choosing generic points on each branch (see points in Fig.~\ref{fig:b2}). All the non-radially symmetric branches are linearly unstable since there exist positive real eigenvalues shown in Table \ref{tab:non-radial}.
\begin{table}[htbp]
	\centering
	\caption{\label{tab:non-radial}The real part of the largest eigenvalues for non-radially symmetric solutions shown in Fig.~\ref{fig:b2}.}
		\begin{tabular}{|c|c|}\hline
			Point & $real(\lambda_{max})$\\
			\hline
			Point 1 & 2.1998$\times 10^{4}$  \\\hline
			Point 2 & 2.3076$\times 10^{4}$\\\hline
			Point 3 & 1.9630$\times 10^{4}$  \\\hline
			Point 4 & 2.2456$\times 10^{4}$  \\\hline
			Point 5 & 2.1340$\times 10^{4}$ \\\hline
			Point 6 & 2.8766$\times 10^{4}$\\\hline
			Point 7 & 2.3730$\times 10^{4}$ \\\hline
			Point 8 & 2.0012$\times 10^{4}$ \\\hline
		\end{tabular}
\end{table}

\section{Conclusion}
In this paper, we develop a mathematical model of plaque formation in early atherosclerosis and describe the geometry change of plaque growth by a free boundary. Then the foam cells, the pressure, and the velocity of plaque moving satisfy a coupled system of PDEs in the intima region $\Omega(t)$. The LDL and HDL cholesterols are two parameters in the free boundary model. Then we solve the free boundary problem analytically in the radially symmetric case and obtain the explicit radially symmetric steady-state solutions for any given parameters. We also analyze the bifurcation points of the free boundary problem and show that there exists a sequence of bifurcations to the cholesterol ratio. Moreover, we also prove that radially symmetric solutions are linearly unstable with arbitrary perturbations but linearly stable with radially symmetric perturbation when the cholesterol ratio is not large. We also verify all the theoretical results by conducting numerical simulations. Furthermore, numerical simulations provide local solution structures near each bifurcation point and non-radially symmetric steady-state solutions which are linearly unstable.  This paper shows why the cholesterol ratio is important in the mathematical context and provides an insight into why plaque patterns also contribute to the rupture. More specifically, if the plaque pattern is radially symmetric, then in most cases, the plaque stays stable\cite{pmid20627248} since the perturbation in the artery is symmetric due to the blood pressure\cite{pmid26468262}. However, if the plaque pattern is irregular, even for a small cholesterol ratio, the plaque goes unstable\cite{pmid7980706}.

\section*{acknowledgments}
	This work is supported by the American Heart Association (Grant 17SDG33660722) and the National
	Science Foundation (Grant DMS-1818769).

\section*{Data Availability Statement}
The data that supports the findings of this study are available within the article [and its supplementary material].

\appendix	
\section{A numerical scheme to approximate $\frac{\partial^2 G}{\partial \theta^2}$}
We use the following finite difference scheme to approximate $\frac{\partial^2 G}{\partial \theta^2}$:
\begin{equation}
\label{ddtheta}
\begin{aligned}
\frac{\partial^2G}{\partial \theta^2}(r_{i,j},\theta_j)  &= a_1 G(r_{i,j},\theta_j) + a_2 G(r_{i+1,j},\theta_j) + a_3 G(r_{i-1,j},\theta_j) \\
&+a_4 G(r_{i,j+1},\theta_{j+1}) + a_5 G(r_{i+1,j+1},\theta_{j+1}) \\
&+ a_6 G(r_{i-1,j+1},\theta_{j+1}) +a_7 G(r_{i,j-1},\theta_{j-1})\\
& + a_8 G(r_{i+1,j-1},\theta_{j-1})+ a_9 G(r_{i-1,j-1},\theta_{j-1}), \\
\end{aligned}
\end{equation}
where
\begin{equation*}
\left\{
\begin{aligned}
a_2 &= \frac{h_{j+1}^2(r_{i,j}-r_{i,j+1})+h_{j-1}^2(r_{i,j}-r_{i,j-1})}{3h_j^2(r_{i,j+1}-2r_{i,j}+r_{i,j-1})}\\
&+ \frac{(r_{i,j+1}-r_{i,j})^3+(r_{i,j-1}-r_{i,j})^3}{3h_j^2(r_{i,j+1}-2r_{i,j}+r_{i,j-1})}\\
a_1 &= -2 - 2a_2,\ a_3 = a_2,\\
a_5 &= -\frac{h_{j+1}^2(r_{i,j}-r_{i,j+1})  + h_{j-1}^2(r_{i,j}-r_{i,j-1})}{6h_{j+1}^2(r_{i,j+1}-2r_{i,j}+r_{i,j-1})}\\
&-\frac{(2r_{i,j}-r_{i,j+1}-r_{i,j-1})^2(r_{i,j}-2r_{i,j+1}+r_{i,j-1})}{6h_{j+1}^2(r_{i,j+1}-2r_{i,j}+r_{i,j-1})} ,\\
&-\frac{3h_{j+1}(r_{i,j}-r_{i,j+1})(2r_{i,j}-r_{i,j+1}-r_{i,j-1})}{6h_{j+1}^2(r_{i,j+1}-2r_{i,j}+r_{i,j-1})}\\
a_6 &= -\frac{h_{j+1}^2(r_{i,j}-r_{i,j+1})  + h_{j-1}^2(r_{i,j}-r_{i,j-1})}{6h_{j+1}^2(r_{i,j+1}-2r_{i,j}+r_{i,j-1})}\\
&-\frac{(2r_{i,j}-r_{i,j+1}-r_{i,j-1})^2(r_{i,j}-2r_{i,j+1}+r_{i,j-1})}{6h_{j+1}^2(r_{i,j+1}-2r_{i,j}+r_{i,j-1})} ,\\
&+\frac{3h_{j+1}(r_{i,j}-r_{i,j+1})(2r_{i,j}-r_{i,j+1}-r_{i,j-1})}{6h_{j+1}^2(r_{i,j+1}-2r_{i,j}+r_{i,j-1})}\\
a_4 & = 1-a_5-a_6,\\
a_8 &= -\frac{h_{j+1}^2(r_{i,j}-r_{i,j+1})+h_{j-1}^2(r_{i,j}-r_{i,j-1})}{6h_{j-1}^2(r_{i,j+1}-2r_{i,j}+r_{i,j-1})}\\
&-\frac{(2r_{i,j}-r_{i,j+1}-r_{i,j-1})^2(r_{i,j}-2r_{i,j-1}+r_{i,j+1})}{6h_{j-1}^2(r_{i,j+1}-2r_{i,j}+r_{i,j-1})} ,\\
& -\frac{3h_{j-1}(r_{i,j}-r_{i,j-1})(2r_{i,j}-r_{i,j+1}-r_{i,j-1})}{6h_{j-1}^2(r_{i,j+1}-2r_{i,j}+r_{i,j-1})}\\
a_9 &= -\frac{h_{j+1}^2(r_{i,j}-r_{i,j+1})+h_{j-1}^2(r_{i,j}-r_{i,j-1})}{6h_{j-1}^2(r_{i,j+1}-2r_{i,j}+r_{i,j-1})}\\
&-\frac{(2r_{i,j}-r_{i,j+1}-r_{i,j-1})^2(r_{i,j}-2r_{i,j-1}+r_{i,j+1})}{6h_{j-1}^2(r_{i,j+1}-2r_{i,j}+r_{i,j-1})} ,\\
& +\frac{3h_{j-1}(r_{i,j}-r_{i,j-1})(2r_{i,j}-r_{i,j+1}-r_{i,j-1})}{6h_{j-1}^2(r_{i,j+1}-2r_{i,j}+r_{i,j-1})}\\
a_7 & = 1-a_8-a_9.
\end{aligned}
\right.
\end{equation*}

\bibliographystyle{siam}  

\bibliography{HDL}

\end{document}